\documentclass[12pt,reqno]{amsart}
\usepackage{amsmath, amsthm, amssymb}

\advance \topmargin by -\headheight
\advance \topmargin by -\headsep
     
\evensidemargin \oddsidemargin
\newtheorem{theorem}{Theorem}[section]
\newtheorem{lemma}[theorem]{Lemma}
\newtheorem{corollary}[theorem]{Corollary}

\theoremstyle{definition}

\theoremstyle{remark}

\numberwithin{equation}{section}

\newcommand{\mmod}[1]{\,\,(\text{mod}\,\,#1)}

   \def\fhat{{\widehat f}}

\def\bfh{{\mathbf h}}

\def\bfk{{\mathbf k}}

\def\bfx{{\mathbf x}}
\def\bfy{{\mathbf y}}

\def\dbN{{\mathbb N}}

\def\dbQ{{\mathbb Q}}

\def\grp{{\mathfrak p}}

\def\bet{{\beta}}  
\def\gam{{\gamma}} \def\gamtil{{\widetilde \gamma}} 
\def\del{{\delta}} 
\def\zet{{\zeta}}

\def\lam{{\lambda}}

\def\Ups{{\Upsilon}} 
\def\phihat{{\widehat \varphi}}

\def\ome{{\omega}} \def\Ome{{\Omega}} 

\def\eps{\varepsilon}

\def\le{\leqslant} \def\ge{\geqslant}

\begin{document}
\title[Waring's problem]{On Waring's problem: some consequences\\ of golubeva's method}
\author[T. D. Wooley]{Trevor D. Wooley}
\address{School of Mathematics, University of Bristol, University Walk, Clifton, Bristol BS8 1TW, United Kingdom}
\email{matdw@bristol.ac.uk}
\date{}
\begin{abstract} We investigate sums of mixed powers involving two squares, two cubes, and various higher powers, concentrating on situations inaccessible to the Hardy-Littlewood method.\end{abstract}
\maketitle

\section{Introduction} Problems of Waring-type involving mixed sums of powers provide convenient specimens on which to test analytic methods that one hopes subsequently to apply more generally. Let $t\ge 2$ and $2\le k_1\le \ldots \le k_t$ be fixed integers, and let $n$ be a natural number sufficiently large in terms of $t$ and $\bfk$. We refer to the problem of establishing the existence of solutions $\bfx\in \dbN^t$ of the Diophantine equation
\begin{equation}\label{1.1}
n=x_1^{k_1}+x_2^{k_2}+\ldots +x_t^{k_t}
\end{equation}
as the {\it Waring problem} corresponding to the exponent $t$-tuple $\bfk$. Associated with this problem is the parameter $\bet(\bfk)=k_1^{-1}+\ldots +k_t^{-1}$ that determines the difficulty of obtaining solutions. Aficionados of the Hardy-Littlewood method will recognise that when $k_t\ge 3$, the intrinsic convexity barrier prevents the circle method from establishing the existence of solutions to (\ref{1.1}) when $\bet(\bfk)\le 2$. The goal of this paper is to resolve a wide class of such problems of Waring-type, if necessary by assuming the truth of the Generalised Riemann Hypothesis. Henceforth we abbreviate the latter hypothesis to GRH, by which we mean the Riemann Hypothesis for all $L$-functions associated with Dirichlet characters. Our first result illustrates what we have in mind, and concerns the exponent tuple $\bfk=(2,2,3,3,6,6)$, for which $\bet(\bfk)=2$.

\begin{theorem}\label{theorem1.1}
Assume the truth of GRH. Then all sufficiently large natural numbers $n$ are represented in the form
\begin{equation}\label{1.2}
x_1^2+x_2^2+x_3^3+x_4^3+x_5^6+x_6^6=n,
\end{equation}
with $x_i\in \dbN$ $(1\le i\le 6)$.
\end{theorem}

As a consequence of recent work of the author \cite{Woo2012} concerning the anticipated asymptotic formula associated with this representation problem, it follows without condition that the number of integers $n$ with $1\le n\le X$, which fail to be represented in the shape (\ref{1.2}), is at most $O((\log X)^{3+\eps})$. The closest analogue to Theorem \ref{theorem1.1} thus far achieved in the literature is the implicit resolution by Vaughan \cite{Vau1980} of the Waring problem associated to the tuple $\bfk=(2,2,3,3,5,5)$, although it would not be difficult to extract from this result a conclusion for the tuple $\bfk=(2,2,3,3,5,6)$. There are very few instances in which the Waring problem corresponding to exponent $\bfk$ has been resolved when $\bet(\bfk)\le 2$. Subject to natural conditions of congruential type, Gauss \cite{Gau1801} tackled the case $\bfk=(2,2,2)$ with $\bet(\bfk)=\tfrac{3}{2}$, and more recently Linnik \cite{Lin1972} and Hooley \cite{Hoo1981b} successfully considered $\bfk=(2,2,3,3,3)$ with $\bet(\bfk)=2$. Also, work of Golubeva \cite{Gol1996, Gol2008} addresses the mixed exponent $\bfk=(2,2,3,3,4,16,4k+1)$ with $\bet(\bfk)=2-\tfrac{1}{48}+\tfrac{1}{4k+1}$. Most recently of all, subject to the validity of the Elliott-Halberstam conjecture together with GRH, work of Friedlander and the author \cite{FW2012} resolves the Waring problem corresponding to the exponent tuple $\bfk=(2,2,4,4,4,k)$, with $\bet(\bfk)=2-\tfrac{1}{4}+\tfrac{1}{k}$. Here, the final $k$th power may be deleted if one is prepared to accommodate certain congruence conditions modulo $48$.\par

Rather than apply the circle method to establish Theorem \ref{theorem1.1}, which as we have noted is limited to situations with $\bet(\bfk)>2$ by the convexity barrier, we instead apply a method of Golubeva involving the theory of ternary quadratic forms \cite{Gol1996, Gol2008}. In applying this method, we avoid in this paper certain difficulties arising from auxiliary congruences by assuming where necessary the truth of GRH. Following some preliminary discussion in \S2, we establish in \S3 the more widely applicable conclusion contained in the following theorem. Here and in \S3, we write
\begin{equation}\label{1.3}
\gam(\bfk)=\prod_{j=1}^t\left(1-\frac{1}{k_i}\right)\quad \text{and}\quad \gamtil(\bfk)=\left( 1-\frac{1}{k_t}\right)\prod_{j=1}^{t-2}\left(1-\frac{1}{k_i}\right).
\end{equation}
We emphasise that throughout this paper, except where otherwise indicated, we assume that $2\le k_1\le \ldots \le k_t$.

\begin{theorem}\label{theorem1.2}
Assume the truth of GRH. Then provided that $\gam(\bfk)<\frac{12}{17}$, all sufficiently large natural numbers $n$ are represented in the form
\begin{equation}\label{1.4}
x_1^2+x_2^2+x_3^3+x_4^3+\sum_{j=1}^ty_j^{k_j}=n,
\end{equation}
with $\bfx\in \dbN^4$ and $\bfy\in \dbN^t$. The same conclusion holds without the assumption of GRH if either (i) one has $t\ge 2$ and $\gamtil(\bfk)<\frac{74}{105}$, or (ii) one has $\gam(\bfk)<\frac{74}{105}$ and the exponents $k_1,\ldots ,k_t$ are not all even.
\end{theorem}

The conclusion of Theorem \ref{theorem1.1} follows at once on noting that $\gam(6,6)=\frac{25}{36}<\frac{12}{17}$. Further corollaries follow likewise with a modicum of computation.

\begin{corollary}\label{corollary1.4}
All sufficiently large natural numbers $n$ are represented as a sum of positive integral powers in the form
$$x_1^2+x_2^2+x_3^3+x_4^3+x_5^5+x_6^8=n,$$
and also in the form
$$x_1^2+x_2^2+x_3^3+x_4^3+x_5^9+x_6^9+x_7^9=n.$$
\end{corollary}

\begin{corollary}\label{corollary1.3}
Assume the truth of GRH. Then all sufficiently large natural numbers $n$ are represented as a sum of positive integral powers in the form
$$x_1^2+x_2^2+x_3^3+x_4^3+x_5^6+x_6^{12}+x_7^{12}=n.$$
\end{corollary}

If one is prepared to assume the Ramanujan Conjecture concerning the Fourier coefficients of cusp forms of weight $\tfrac{3}{2}$, then further progress is possible, as we demonstrate in \S4.

\begin{theorem}\label{theorem1.5}
Assume the truth of GRH and the Ramanujan Conjecture. Then provided that $\gam(\bfk)<\frac{5}{6}$, all sufficiently large natural numbers $n$ are represented in the form (\ref{1.4}) with $\bfx\in \dbN^4$ and $\bfy\in \dbN^t$. The same conclusion holds without the assumption of GRH if either (i) one has $t\ge 2$ and $\gamtil(\bfk)<\frac{5}{6}$, or (ii) one has $\gam(\bfk)<\frac{5}{6}$ and the exponents $k_1,\ldots ,k_t$ are not all even.
\end{theorem}

We record two immediate consequences of Theorem \ref{theorem1.5} which require no further explanation.

\begin{corollary}\label{corollary1.6}
Assume the truth of the Ramanujan Conjecture. Then all sufficiently large natural numbers $n$ are represented as a sum of positive integral powers in the form
$$x_1^2+x_2^2+x_3^3+x_4^3+x_5^5=n,$$
and also in the form
$$x_1^2+x_2^2+x_3^3+x_4^3+x_5^6+x_6^{2k+1}=n.$$
\end{corollary}

\begin{corollary}\label{corollary1.7}
Assume the truth of GRH and the Ramanujan Conjecture. Then all sufficiently large natural numbers $n$ are represented as a sum of positive integral powers in the form
$$x_1^2+x_2^2+x_3^3+x_4^3+x_5^4=n.$$
\end{corollary}

We finish by remarking that several authors have considered Waring's problem in the general setting (\ref{1.1}) with mixed powers $k_1,\ldots ,k_t$. The reader will locate a representative slice of this literature in the sources Br\"udern \cite{Bru1987}, Hooley \cite{Hoo1981a} and Vaughan \cite{Vau1971}.\par

It is a pleasure to record here my gratitude to Elena Golubeva and Valentin Blomer for correspondence and conversations several years ago concerning Golubeva's theorem on ternary quadratic forms and its proof, here reported as Theorem \ref{theorem2.1} below.

\section{Golubeva's method: preliminaries}
Our goal in this section and the next is the proof of Theorem \ref{theorem1.2}, together with the conclusion of Theorem \ref{theorem1.1} which essentially amounts to a corollary of the former theorem. Here we make use of a method originating in work of Linnik \cite{Lin1943}, and much enhanced by the level-lowering procedure of Golubeva \cite{Gol1996, Gol2008}. We seek to establish the solubility of the equation (\ref{1.4}) when $n$ is large by putting $x_3=A+x_0$ and $x_4=A-x_0$, with $A\in \dbN$ of size nearly $n^{1/3}$ and $|x_0|<A$. Notice here that when $A$ and $x_0$ are both integers, then $x_3$ and $x_4$ are necessarily of the same parity. In this way, the representation problem (\ref{1.4}) is transformed into the related one
\begin{equation}\label{2.1}
x_1^2+x_2^2+6Ax_0^2=n-2A^3-\sum_{j=1}^ty_j^{k_j},
\end{equation}
in which we seek a solution with $A\in \dbN$, $|x_0|<A$ and $\bfy\in \dbN^t$. Were $A$ to be very small compared to the integer on the right hand side of (\ref{2.1}), for fixed $\bfy$, and $x_0$ unconstrained, then such a problem would be well within the scope of the theory of ternary quadratic forms. The large size of $A$ obstructs such an approach. In order to bypass this difficulty, we instead seek to choose $y_1,\ldots ,y_t$ in such a way that
$$n-\sum_{j=1}^ty_j^{k_j}$$
is divisible by $q^{6h}$, for some natural numbers $q$ and $h$. If $q^{2h}$ is reasonably close in size to $n^{1/3}$, then one may substitute
$$A=q^{6h}B,\quad x_0=z_0,\quad x_i=q^{3h}z_i\quad (i=1,2),$$
$$N=q^{-6h}\Bigl( n-\sum_{j=1}^ty_j^{k_j}\Bigr) -2q^{12h}B^3.$$
This transforms the representation problem (\ref{2.1}) into the shape
$$z_1^2+z_2^2+6Bz_0^2=N,$$
with $N$ now relatively large compared to $B$. In this way, the level of the modular forms involved in the analysis of the ternary quadratic form $z_1^2+z_2^2+6Bz_0^2$ is lowered to the point that a successful analysis again becomes feasible. This level lowering approach of Golubeva is further enhanced by choosing $B$ in such a manner that $N$ contains a sizeable square factor, this innovation also having been devised by Golubeva.\par

The key ingredient from the theory of ternary quadratic forms of which we make use is the following result of Golubeva.

\begin{theorem}\label{theorem2.1}
For each $\del>0$, there exists a positive number $C(\del)$ with the property that, whenever $n\in \dbN$ and $p$ is a prime number satisfying the following conditions:
\item{(i)} one has $(n,6p)=1$;
\item{(ii)} the congruence $n\equiv x^2+y^2+6pz^2\mmod{16}$ is soluble;
\item{(iii)} the integer $n$ may be written in the form $n=tm^2$ with $(t,m)=1$, and one has $nm^{12}>C(\del)p^{21+\del}$;
\vskip.1cm \noindent then the equation
\begin{equation}\label{2.2}
n=x^2+y^2+6pz^2
\end{equation}
has a solution in natural numbers $x$, $y$ and $z$.
\end{theorem}

\begin{proof} This is in all essentials \cite[Theorem 2]{Gol2008}. We have modified the conclusion to assert that the integers $x$, $y$ and $z$ are natural numbers. Such follows from the argument of the proof of \cite[Theorem 2]{Gol2008}, on noting that the number of representations of $n$ in the shape (\ref{2.2}) with $xyz=0$ is plainly $O_\eps(n^\eps)$. Meanwhile, for each $\eps>0$, the total number of solutions delivered by Golubeva's argument is at least $C_\eps n^{1/2-\eps}p^{-1/2}$, for a positive constant $C_\eps$ depending at most on $\eps$.\par

It may be helpful to the reader examining the proof of \cite[Theorem 2]{Gol2008} to think of $m^2$ as the largest square divisor of $n$, so that $t$ is squarefree. The bounds for the $n$-th Fourier coefficients of eigenforms, for arbitrary $n$, reported by Golubeva are then discussed in detail in \cite[Lemma 1]{Blo2008}. See the sketch proof of Theorem \ref{theorem4.1} below for more on this matter.\end{proof}

We next set up the machinery to construct the integer $q$ that plays a key role in the above sketch. This rests on the existence of small primes $p_j$ having the property that a given integer $m$ is a $k_j$-th power residue modulo $p_j$. Current technology fails to deliver such a conclusion in general without appeal to a suitable Riemann Hypothesis.

\begin{lemma}\label{lemma2.2} Let $k,D\in \dbN$, and suppose that $m$ is a large natural number with $m\gg D^A$, for some fixed $A>0$. Then there exists a prime $\varpi$ with $(mD,\varpi)=1$ having the property that the congruence $y^k\equiv m\mmod{\varpi^h}$ is soluble for every natural number $h$. When $k$ is odd, one may choose $\varpi$ with $\varpi\ll_k \log m$. When $k$ is even, meanwhile, provided that the Riemann Hypothesis holds for the Dedekind zeta function $\zet_L(s)$, where $L=\dbQ(e^{2\pi i/k},\sqrt[k]{m})$, one may choose $\varpi$ with $\varpi\ll_{k,\eps} (\log m)^{2+\eps}$.
\end{lemma}

\begin{proof} Consider first the situation in which $k$ is odd. We take $\varpi$ to be any prime number with $\varpi\equiv 2\mmod{k}$ and $\varpi\nmid mD$. Since the number of prime divisors of $mD$ is $O(\log m/\log \log m)$, it follows from the Prime Number Theorem in arithmetic progressions that such a prime exists with $\varpi\ll_k \log m$. Since $(\varpi-1,k)=1$, it follows from Fermat's Little Theorem that the congruence $y^k\equiv m\mmod{\varpi}$ is soluble, and then an application of Hensel's Lemma reveals that for every integer $h$, the congruence $y^k\equiv m\mmod{\varpi^h}$ is soluble. This completes the proof of the lemma when $k$ is odd.\par

When $k$ is even, we must resort to higher reciprocity laws (see Milne \cite[Chapter V, especially p.162]{Mil2011} and Pollack \cite{Pol2012} for suitable background material). Let $\zet_k=e^{2\pi i/k}$, and write $K=\dbQ(\zet_k)$ and $L=\dbQ(\zet_k,\sqrt[k]{m})$. Write $d_L$ for the discriminant of $L$ over $\dbQ$. Then it follows from the Dedekind-Kummer splitting criterion that for any prime ideal $\grp$ of $K$ relatively prime to $k$ and $m$, so that $\grp$ is unramified in $L$, the $k$th power residue symbol $\left({\displaystyle{\frac{m}{\grp}}}\right)_k$ is equal to $1$ if and only if $\grp$ splits in $L$. The latter in turn happens if and only if the Artin symbol $\left( {\displaystyle{\frac{L/K}{\grp}}}\right)$ is the identity in $\text{Gal}(L/K)$. Assuming the Riemann Hypothesis for the Dedekind zeta function $\zet_L(s)$, it follows from work of Lagarias and Odlyzko \cite{LO1977} that there exists a prime ideal $\grp$ of $K$, unramified in $L$, with $\left( {\displaystyle{\frac{L/K}{\grp}}}\right)$ equal to the identity and satisfying
$$N_{K/\dbQ}(\grp)\ll (\log d_L)^2(\log \log d_L)^4\ll_\eps (\log m)^{2+\eps}.$$
This work shows, moreover, that $N_{K/\dbQ}(\grp)$ may be chosen to be a rational prime, say $\varpi$, with $(\varpi,kmD)=1$. Thus we deduce that the congruence $y^k\equiv m\mmod{\varpi}$ is soluble for some rational prime $\varpi$ satisfying
$$\varpi\ll_\eps (\log m)^{2+\eps}\quad \text{and}\quad (kmD,\varpi)=1.$$
An application of Hensel's Lemma again now shows that the congruence $y^k\equiv m\mmod{\varpi^h}$ is soluble for every natural number $h$. This completes the proof of the lemma. 
\end{proof}

We note that the Riemann Hypothesis for the Dedekind zeta function $\zet_L(s)$ occurring in the statement of Lemma \ref{lemma2.2} follows from GRH.\par

We next make some preparations which facilitate the application of Lemma \ref{lemma2.2} in the level-lowering procedure of Golubeva. The first observation that we must keep in mind is that in order to apply Theorem \ref{theorem2.1}, one must ensure that the conditions (i) and (ii) are met. It transpires that some slightly delicate footwork is required as we proceed through the argument so as to ensure that such remains possible. In order to motivate some of this manoeuvring, we note that a modicum of computation confirms that when $p$ is an odd prime, then for any integer $\lam$ the expression
$$x^2+y^2+6pz^2+2\lam^2p^3$$
represents all of the odd residue classes modulo $16$, with one at least of $x$ and $y$ odd. It follows as a consequence of this observation that we may concentrate on condition (i) of Theorem \ref{theorem2.1} in what follows.\par

We now briefly describe the setting for the next lemma. We suppose throughout that $n$ is a sufficiently large natural number. It is convenient to adopt the notation of writing $K$ for the product $k_1\ldots k_{t-1}$. We consider fixed distinct prime numbers $\varpi_1,\ldots ,\varpi_{t-1}$ satisfying the conditions $(30K,\varpi_i)=1$ $(1\le i\le t-1)$ and $\varpi_{t-1}\nmid n$. Since the number of distinct prime divisors of $n$ is $O(\log n/\log \log n)$, we may suppose in addition that
\begin{equation}\label{2.3}
K^{10}<\varpi_{t-1}\ll_K \log n\quad \text{and}\quad K^{10}<\varpi_i\ll_K 1\quad (1\le i\le t-2).
\end{equation}
Finally, we write $\Ome_u=\varpi_1\ldots \varpi_u$, with the familiar convention that $\Ome_0=1$.

\begin{lemma}\label{lemma2.3}
Assume the truth of GRH. Then there exist $y,h\in \dbN$, and a prime number $\varpi$ with $\varpi\le (\log n)^3$, having the following properties:
\item{(a)} $(30n\Ome_{t-1},\varpi)=1$ and $y^{k_t}\equiv n\mmod{\varpi^{6Kh}}$;
\item{(b)} $n^{1/k_t}(\log n)^{-20tK}\le \varpi^{6Kh}\le y\le n^{1/k_t}(\log n)^{-t}$;
\item{(c)} $(n-y^{k_t}, 10\Ome_{t-1})=1$ and $n-y^{k_t}\not\equiv 2\mmod{3}$;
\item{(d)} when $t\ge 2$, then the integer $\varpi^{-6Kh}(n-y^{k_t})$ is a $k_{t-1}$-th power residue modulo $\varpi_{t-1}$.\vskip.1cm
\noindent When $k_t$ is odd, this conclusion is independent of GRH.
\end{lemma}

\begin{proof} When $k_t$ is even, assume the truth of GRH. Then it follows from Lemma \ref{lemma2.2} that when $n$ is sufficiently large, then there exists a prime number $\varpi$ with $\varpi\le (\log n)^3$ and $(30n\Ome_{t-1},\varpi)=1$ such that, for any $l\in \dbN$, there exists a solution of the congruence
\begin{equation}\label{2.4}
z^{k_t}\equiv n\mmod{\varpi^l}.
\end{equation}
We choose $l$ to be the largest multiple of $6K$, say $l=6Kh$, having the property that
$$\varpi^l\le n^{1/k_t}(\log n)^{-t-2}.$$

\par Let $z$ denote any solution of (\ref{2.4}) with $1\le z\le \varpi^l$. Also, define the integer $g_3(n)$ by putting
$$g_3(n)=\begin{cases} 0,&\text{when $n\not\equiv 2\mmod{3}$,}\\
1,&\text{when $n\equiv 2\mmod{3}$.}\end{cases}$$
When $t\ge 2$ it follows from Weil \cite{Wei1949} that, since $(n\varpi,\varpi_{t-1})=1$, the congruence
$$\varpi^{-l}n\equiv w^{k_{t-1}}+\varpi^{-l}v^{k_t}\mmod{\varpi_{t-1}}$$
possesses a solution $w,v$ with $\varpi_{t-1}\nmid wv$. When $t\ge 2$, fix any one such solution $w,v$. In addition, when $q\in \{2,5,\varpi_1,\ldots ,\varpi_{t-2}\}$, put
$$g_q(n)=\begin{cases} 0,&\text{when $(n,q)=1$,}\\
1,&\text{when $q|n$.}\end{cases}$$
Then since $(30\Ome_{t-1},\varpi)=1$, it follows from the Chinese Remainder Theorem that there exists an integer $g$ with $1\le g\le 30\Ome_{t-1}$ having the property that
$$z+g\varpi^l\equiv g_q(n)\mmod{q}\quad (q=2,3,5,\varpi_1,\ldots ,\varpi_{t-2}),$$
and such that when $t\ge 2$ one has
$$z+g\varpi^l\equiv v\mmod{\varpi_{t-1}}.$$
We consequently have
$$(n-(z+g\varpi^l)^{k_t},10\Ome_{t-1})=1\quad \text{and}\quad n-(z+g\varpi^l)^{k_t}\not\equiv 2\mmod{3}.$$
In addition, when $t\ge 2$ the congruence
$$\varpi^{-l}(n-(z+g\varpi^l)^{k_t})\equiv w^{k_{t-1}}\mmod{\varpi_{t-1}}$$
is soluble with $\varpi_{t-1}\nmid w$.\par

We now take $y=z+g\varpi^l$, and note that from (\ref{2.3}) we have the bounds
$$y\le 31\Ome_{t-1}\varpi^l\le (\log n)^2(n^{1/k_t}(\log n)^{-t-2})\le n^{1/k_t}(\log n)^{-t}$$
and
$$y\ge \varpi^l\ge n^{1/k_t}(\log n)^{-t-2-18K}\ge n^{1/k_t}(\log n)^{-20tK}.$$
The conclusion of the lemma is now immediate.
\end{proof}

An application of Lemma \ref{lemma2.3} sows the seeds for the iteration of a result with a similar flavour.

\begin{lemma}\label{lemma2.4}
Let $u$ be a natural number with $1\le u\le t-1$, and let $m\in \dbN$ satisfy $\log m\gg \log n$. Suppose that $m$ is a $k_u$-th power residue modulo $\varpi_u$, and in particular that $\varpi_u\nmid m$. In addition, when $u\ge 2$ suppose that $\varpi_{u-1}\nmid m$. Then there exist natural numbers $y$ and $h$ having the following properties:
\item{(a)} $y^{k_u}\equiv m \mmod{\varpi_u^{6Kh}}$;
\item{(b)} $m^{1/k_u}(\log m)^{-20uk}\le \varpi_u^{6Kh}\le y\le m^{1/k_u}(\log m)^{-u}$;
\item{(c)} $(m-y^{k_u},10\Ome_{u-1})=1$ and $m-y^{k_u}\not\equiv 2\mmod{3}$;
\item{(d)} when $u\ge 2$, then the integer $\varpi_u^{-6Kh}(m-y^{k_u})$ is a $k_{u-1}$-th power residue modulo $\varpi_{u-1}$.
\end{lemma}

\begin{proof} Since $ m$ is a $k_u$-th power residue modulo $\varpi_u$ and $(\varpi_u,k_u)=1$, it follows from an application of Hensel's Lemma that, for any $l\in \dbN$, there exists a solution of the congruence
\begin{equation}\label{2.5}
z^{k_u}\equiv m\mmod{\varpi_u^l}.
\end{equation}
We choose $l$ to be the largest multiple of $6K$, say $l=6Kh$, having the property that
$$\varpi_u^l\le m^{1/k_u}(\log m)^{-u-1}.$$

\par Let $z$ denote any solution of (\ref{2.5}) with $1\le z\le \varpi_u^l$. We define the integers $g_q(m)$ just as in the argument of the proof of Lemma \ref{lemma2.3} when $q=2,3,5,\varpi_1,\ldots ,\varpi_{u-2}$, mutatis mutandis. When $u\ge 2$ it follows from Weil \cite{Wei1949} that, since $(m\varpi_u,\varpi_{u-1})=1$, the congruence
$$\varpi_u^{-l}m\equiv w^{k_{u-1}}+\varpi_u^{-l}v^{k_u}\mmod{\varpi_{u-1}}$$
possesses a solution $w,v$ with $\varpi_{u-1}\nmid wv$. When $u\ge 2$, we fix any one such solution $w,v$. Since $(30\Ome_{u-1},\varpi_u)=1$, it follows from the Chinese Remainder Theorem that there exists an integer $g$ with $1\le g\le 30\Ome_{u-1}$ having the property that
$$z+g\varpi_u^l\equiv g_q(m)\mmod{q}\quad (q=2,3,5,\varpi_1,\ldots ,\varpi_{u-2})$$
and such that when $u\ge 2$ one has
$$z+g\varpi_u^l\equiv v\mmod{\varpi_{u-1}}.$$
We consequently have
$$(m-(z+g\varpi_u^l)^{k_u},10\Ome_{u-1})=1\quad \text{and}\quad m-(z+g\varpi_u^l)^{k_u}\not\equiv 2\mmod{3}.$$
In addition, when $u\ge 2$, the congruence
$$\varpi_u^{-l}(m-(z+g\varpi_u^l)^{k_u})\equiv w^{k_{u-1}}\mmod{\varpi_{u-1}}$$
is soluble with $\varpi_{u-1}\nmid w$.\par

We now take $y=z+g\varpi_u^l$, and note that from (\ref{2.3}) we have the bounds
$$y\le 31\Ome_{u-1}\varpi_u^l\le (\log m)(m^{1/k_u}(\log m)^{-u-1})\le m^{1/k_u}(\log m)^{-1}$$
and
$$y\ge \varpi_u^l\ge m^{1/k_u}(\log m)^{-u-1-K}\ge m^{1/k_u}(\log m)^{-20uK}.$$
The conclusion of the lemma now follows.
\end{proof}

Finally, in order to avoid the assumption of GRH when the exponents $k_1,\ldots ,k_t$ are all even, we require an additional device. We recall our implicit assumption that $k_{t-1}\le k_t$.

\begin{lemma}\label{lemma2.5} Suppose that $t\ge 2$, and let $n$ be a sufficiently large natural number with $\varpi_{t-1}\nmid n$. Then there exist natural numbers $y_1$, $y_2$ and $h$ having the following properties:
\item{(a)} $y_1^{k_{t-1}}+y_2^{k_t}\equiv n\mmod{\varpi_{t-1}^{6Kh}}$;
\item{(b)} $n^{1/k_t}(\log n)^{-20tK}\le \varpi_{t-1}^{6Kh}\le \min\{y_1,y_2\}\le \max\{y_1,y_2\}\le n^{1/k_t}(\log n)^{-t}$;
\item{(c)} $(n-y_1^{k_{t-1}}-y_2^{k_t},10\Ome_{t-2})=1$ and $n-y_1^{k_{t-1}}-y_2^{k_t}\not\equiv 2\mmod{3}$;
\item{(d)} when $t\ge 3$, then the integer $\varpi_{t-1}^{-6Kh}(n-y_1^{k_{t-1}}-y_2^{k_t})$ is a $k_{t-2}$-th power residue modulo $\varpi_{t-2}$.\vskip.1cm
\end{lemma}

\begin{proof} Since $\varpi_{t-1}\nmid n$, it follows from Weil \cite{Wei1949} that the congruence
$$n\equiv u^{k_{t-1}}+v^{k_t}\mmod{\varpi_{t-1}}$$
possesses a solution $u,v$ with $\varpi_{t-1}\nmid uv$. Since $(k_{t-1},\varpi_{t-1})=1$, an application of Hensel's Lemma with $v$ fixed shows that, for any $l\in \dbN$, there exists a solution of the congruence
\begin{equation}\label{2.6}
n\equiv z^{k_{t-1}}+w^{k_t}\mmod{\varpi_{t-1}^l},
\end{equation}
with
$$z\equiv u\mmod{\varpi_{t-1}}\quad \text{and}\quad w\equiv v\mmod{\varpi_{t-1}}.$$
We choose $l$ to be the largest multiple of $6K$, say $l=6Kh$, having the property that
$$\varpi_{t-1}^l\le n^{1/k_t}(\log n)^{-t-2}.$$
Let $z,w$ denote any solution of (\ref{2.6}) with $\varpi_{t-1}^l<z,w\le 2\varpi_{t-1}^l$. When $t\ge 3$, we may proceed as in the argument of the proof of Lemma \ref{lemma2.3} to show that the congruence
$$\varpi_{t-1}^{-l}(n-w^{k_t})\equiv x^{k_{t-2}}+\varpi_{t-1}^{-l}y^{k_{t-1}}\mmod{\varpi_{t-2}}$$
possesses a solution $x,y$ with $\varpi_{t-2}\nmid xy$. Note here that
$$(n-w^{k_t},\varpi_{t-1})=(z^{k_{t-1}},\varpi_{t-1})=(u^{k_{t-1}},\varpi_{t-1})=1.$$

\par When $t\ge 3$, fix any one such solution $x,y$. In addition, when $1\le i<t-2$, define
$$g_{\varpi_i}(n)=\begin{cases} 0,&\text{when $(n-w^{k_t},\varpi_i)=1$,}\\
1,&\text{when $\varpi_i|(n-w^{k_t})$,}\end{cases}$$
and define $g_q(n)$ likewise when $q=2,5$. When $q=3$ we define $g_3(n)$ by putting
$$g_3(n)=\begin{cases} 0,&\text{when $n-w^{k_t}\not\equiv 2\mmod{3}$,}\\
1,&\text{when $n-w^{k_t}\equiv 2\mmod{3}$.}\end{cases}$$
Then since $(30\Ome_{t-2},\varpi_{t-1})=1$, it follows from the Chinese Remainder Theorem that there exists an integer $g$ with $1\le g\le 30\Ome_{t-2}$ having the property that
$$z+g\varpi_{t-1}^l\equiv g_q(n)\mmod{q}\quad (q=2,3,5,\varpi_1,\ldots ,\varpi_{t-3}),$$
and such that when $t\ge 3$ one has
$$z+g\varpi_{t-1}^l\equiv y\mmod{\varpi_{t-2}}.$$
We consequently have
$$(n-w^{k_t}-(z+g\varpi_{t-1}^l)^{k_{t-1}},10\Ome_{t-2})=1$$
and
$$n-w^{k_t}-(z+g\varpi_{t-1}^l)^{k_{t-1}}\not\equiv 2\mmod{3}.$$
In addition, when $t\ge 3$, the congruence
$$\varpi_{t-1}^{-l}(n-w^{k_t}-(z+g\varpi_{t-1}^l)^{k_{t-1}})\equiv u^{k_{t-2}}\mmod{\varpi_{t-2}}$$
is soluble with $\varpi_{t-2}\nmid u$.\par

We now take $y_1=z+g\varpi_{t-1}^l$ and $y_2=w$, and note that
$$\max\{y_1,y_2\}\le 31\Ome_{t-2}\varpi_{t-1}^l\le (\log n)^2(n^{1/k_t}(\log n)^{-t-2})=n^{1/k_t}(\log n)^{-t}$$
and
$$\min\{y_1,y_2\}\ge \varpi_{t-1}^l\ge n^{1/k_t}(\log n)^{-t-2-6K}\ge n^{1/k_t}(\log n)^{-20tK}.$$
The conclusion of the lemma is now immediate.
\end{proof}

\section{Golubeva's method: an iterative process}
We now proceed in an iterative fashion. We focus in our main discussion on the most interesting situations in which $\gam(\bfk)>\tfrac{2}{3}$, though later we sketch how to modify this central argument so as to handle the easier cases in which $\gam(\bfk)\le \tfrac{2}{3}$. In addition, our initial focus is on the situation where either (i) not all of the exponents $k_1,\ldots ,k_t$ are even, or (ii) one assumes the truth of GRH. In the first situation, we relabel exponents so that $k_t$ is odd, so that in our application of Lemma \ref{lemma2.3} we are able to avoid the assumption of GRH.\par

Consider the representation problem
\begin{equation}\label{3.1}
m_t=v_t^2+w_t^2+(A_t+z)^2+(A_t-z)^2+\sum_{j=1}^ty_{jt}^{k_j},
\end{equation}
which we seek to solve in natural numbers $v_t$, $w_t$, $A_t$, $z$, $\bfy$ subject to the condition $A_t>(n/6)^{1/3}$, whenever $m_t=n$ is sufficiently large. Given such a solution of (\ref{3.1}), one has
$$m_t=v_t^2+w_t^2+6A_tz^2+2A_t^3+\sum_{j=1}^ty_{jt}^{k_j},$$
and hence
$$z<(\tfrac{1}{6}n/A_t)^{1/2}<(n/6)^{1/3}<A_t,$$
so that (\ref{3.1}) exhibits a solution of the Waring problem (\ref{1.4}) in natural numbers $\bfx,\bfy$.\par

Recall that $K=k_1\ldots k_{t-1}$, and consider fixed distinct prime numbers $\varpi_1,\ldots ,\varpi_{t-1}$ with $(30K,\varpi_i)=1$ $(1\le i\le t-1)$ and $\varpi_{t-1}\nmid n$, and satisfying (\ref{2.3}). From Lemma \ref{lemma2.3}, we find that there exist natural numbers $y=y_{tt}$ and $h=h_t$, and a prime number $\varpi_t$ with $\varpi_t\le (\log n)^3$, having the properties (a) to (d) of the conclusion of that lemma. Here, we assume GRH only in circumstances in which $k_t$ (and indeed every exponent $k_i$) is even.\par

We introduce some notation in order to assist with the iteration to come. When $0\le r\le t$, define $\gam_r$ and $\Ups_r=\Ups_r(\bfh)$ by
$$\gam_r=\prod_{r<l\le t}\left(1-\frac{1}{k_l}\right) \quad \text{and}\quad \Ups_r=\prod_{r<l\le t}\varpi_l^{6Kh_l}.$$
Then with the integers $y_{tt}$ and $h_t$, and the prime number $\varpi_t$, fixed as above, we find that the representation problem (\ref{3.1}) may be solved whenever the derived representation problem
\begin{align}
m_t-y_{tt}^{k_t}=(\varpi_t^{3Kh_t}v_{t-1})^2+&(\varpi_t^{3Kh_t}w_{t-1})^2+
6(\varpi_t^{6Kh_t}A_{t-1})z^2\notag \\
&+2(\varpi_t^{6Kh_t}A_{t-1})^3+\sum_{j=1}^{t-1}(\varpi_t^{6h_tK/k_j}y_{j,t-1})^{k_j}\label{3.2}
\end{align}
is soluble in natural numbers $v_{t-1},w_{t-1},A_{t-1},z$ and $y_{j,t-1}$ $(1\le j\le t-1)$, in which we impose the condition
$$\varpi_t^{6Kh_t}A_{t-1}>(n/6)^{1/3}.$$

\par Write
$$m_{t-1}=\varpi_t^{-6Kh_t}(m_t-y_{tt}^{k_t}).$$
Then in view of property (a) of Lemma \ref{lemma2.3}, the representation problem (\ref{3.2}) is equivalent to
\begin{equation}\label{3.3}
m_r=v_r^2+w_r^2+6A_rz^2+2\Ups_r^2A_r^3+\sum_{j=1}^ry_{jr}^{k_j},
\end{equation}
subject to
\begin{equation}\label{3.4}
\Ups_rA_r>(n/6)^{1/3},
\end{equation}
in the special case $r=t-1$. Notice here, again in the special case $r=t-1$, that property (b) of Lemma \ref{lemma2.3} ensures that
\begin{equation}\label{3.5}
n^{1-\gam_r}(\log n)^{-20(t-r)tK}\le \Ups_r\le n^{1-\gam_r}(\log n)^{20(t-r)tK}
\end{equation}
and
\begin{equation}\label{3.6}
n(1-2(t-r)(\log n)^{-1})\le \Ups_rm_r\le n.
\end{equation}
Thus, in particular, one has
\begin{equation}\label{3.7}
\tfrac{1}{2}n^{\gam_r}(\log n)^{-20(t-r)tK}\le m_r\le n^{\gam_r}(\log n)^{20(t-r)tK}.
\end{equation}
On noting that $\varpi_{r+1}^{6Kh_{r+1}}\equiv 1\mmod{3}$, properties (a) and (c) furnish the relations
\begin{equation}\label{3.8}
(m_r,10\Ome_r)=1\quad \text{and}\quad m_r\not\equiv 2\mmod{3}.
\end{equation}
In addition, property (d) shows that when $t\ge 2$, the congruence
\begin{equation}\label{3.9}
m_r\equiv y^{k_r}\mmod{\varpi_r}
\end{equation}
is soluble for some integer $y$ with $(y,\varpi_r)=1$.\par

We have shown that the Waring problem (\ref{1.4}) is soluble whenever, for $r=t-1$, the representation problem (\ref{3.3}) is soluble subject to the conditions (\ref{3.4}) to (\ref{3.9}). We now show by induction that the same is true for $0\le r<t-1$. Let $u$ be an integer with $1\le u\le t-1$, and suppose that the inductive hypothesis holds for $r=u$. The condition (\ref{3.7}) ensures that $\log m_u\gg \log n$, and (\ref{3.8}) ensures that $(m_u,\varpi_{u-1}\varpi_u)=1$. In addition, it follows from (\ref{3.9}) that $m_u$ is a $k_u$-th power residue modulo $\varpi_u$. Then we deduce from Lemma \ref{lemma2.4} that there exist natural numbers $y=y_{uu}$ and $h=h_u$, having the properties (a) to (d) of the conclusion of that lemma. As in the discussion corresponding to the case $r=t$, it follows that with the integers $y_{uu}$ and $h_u$ fixed, the representation problem (\ref{3.3}) may be solved when $r=u$ provided that the problem
\begin{align}
m_u-y_{uu}^{k_u}=(\varpi_u^{3Kh_u}&v_{u-1})^2+(\varpi_u^{3Kh_u}w_{u-1})^2+
6(\varpi_u^{6Kh_u}A_{u-1})z^2\notag \\
&+2\Ups_u^2(\varpi_u^{6Kh_u}A_{u-1})^3+\sum_{j=1}^{u-1}(\varpi_u^{6h_uK/k_j}y_{j,u-1})^{k_j}\label{3.10}
\end{align}
is soluble in natural numbers $v_{u-1},w_{u-1},A_{u-1},z$ and $y_{j,u-1}$ $(1\le j\le u-1)$, in which we impose the condition
$$\Ups_{u-1}A_{u-1}>(n/6)^{1/3}.$$

\par Write
\begin{equation}\label{3.11}
m_{u-1}=\varpi_u^{-6Kh_u}(m_u-y_{uu}^{k_u}).
\end{equation}
Then by property (a) of Lemma \ref{lemma2.4}, the representation problem (\ref{3.10}) is equivalent to (\ref{3.3}) subject to (\ref{3.4}) in the special case $r=u-1$. Continuing our restriction to the special case $r=u-1$, property (b) of Lemma \ref{lemma2.4} combines with the bounds (\ref{3.6}) and (\ref{3.7}) to deliver the estimates
\begin{align*}
\Ups_{u-1}&=\varpi_u^{6Kh_u}\Ups_u\le m_u^{1/k_u}(\log m_u)^{-u}\Ups_u\le m_u^{1/k_u-1}(\log m_u)^{-u}(m_u\Ups_u)\\
&\le n\left( \tfrac{1}{2}n^{\gam_u}(\log n)^{-20(t-u)tK}\right)^{1/k_u-1}\le 2n^{1-\gam_{u-1}}(\log n)^{20(t-u)tK}\\
&\le n^{1-\gam_{u-1}}(\log n)^{20(t-u+1)tK}
\end{align*}
and
\begin{align*}
\Ups_{u-1}&=\varpi_u^{6Kh_u}\Ups_u\ge m_u^{1/k_u-1}(\log m_u)^{-20uK}(m_u\Ups_u)\\
&\ge \tfrac{1}{2}n\left( n^{\gam_u}(\log n)^{20(t-u)tK}\right)^{1/k_u-1}(\log n)^{-20uK}\\
&\ge n^{1-\gam_{u-1}}(\log n)^{-20(t-u+1)tK},
\end{align*}
so that (\ref{3.5}) holds with $r=u-1$. Likewise, one finds that property (b) in combination with (\ref{3.6}) and (\ref{3.11}) shows that
$$\Ups_{u-1}m_{u-1}\le \Ups_um_u\le n$$
and
$$\Ups_{u-1}m_{u-1}\ge \Ups_um_u(1-(\log m_u)^{-u})\ge n(1-2(t-u+1)(\log n)^{-1}),$$
so that (\ref{3.6}) holds with $r=u-1$. Finally, property (c) ensures that (\ref{3.8}) holds with $r=u-1$, and property (d) shows that when $u\ge 2$, then (\ref{3.9}) holds with $r=u-1$. We have therefore confirmed the inductive hypothesis with $r=u-1$, and this completes the inductive step.\par

At this stage we have shown that the Waring problem (\ref{1.4}) is soluble provided that the representation problem (\ref{3.3}) is soluble when $r=0$ subject to the conditions (\ref{3.4}) to (\ref{3.8}). Simplifying notation in the obvious manner, it therefore suffices to consider the representation problem
\begin{equation}\label{3.12}
m=v^2+w^2+6Bz^2+2\lam^2B^3,
\end{equation}
where $m$ is a sufficiently large positive integer with $(m,10)=1$ and $m\not\equiv 2\mmod{3}$. We may suppose further that $(\lam,30m)=1$, that $\lam$ is an odd square, and that
\begin{equation}\label{3.13}
\lam^{-1}n(1-2t(\log n)^{-1})\le m\le \lam^{-1}n,
\end{equation}
and
\begin{equation}\label{3.14}
n^{1-\gam_0}(\log n)^{-20t^2K}\le \lam\le n^{1-\gam_0}(\log n)^{20t^2K}.
\end{equation}
Finally, our goal is to find a solution satisfying the condition $\lam B>(n/6)^{1/3}$.\par

Consider the congruence
$$2\lam^2B^3\equiv m\mmod{5}.$$
Since $5\nmid m$, and $5$ is distinct from $2,\varpi_1,\ldots ,\varpi_t$, and is congruent to $2$ modulo $3$, the congruence
$$B^3\equiv (2\lam^2)^{-1}m\mmod{5}$$
is soluble with $(B,m)=1$. An application of Hensel's Lemma consequently shows that the congruence
\begin{equation}\label{3.15}
2\lam^2B^3\equiv m\mmod{5^{2h}}
\end{equation}
is soluble for every natural number $h$. Let $\eps$ be a sufficiently small positive number, and take $c=2+2\eps$ when GRH holds, and otherwise put $c=\tfrac{12}{5}+2\eps$. We take $h$ to be the largest integer for which one has
$$\lam (5^{2h+1})^c<(n/6)^{1/3},$$
so that
\begin{equation}\label{3.16}
5^{-2c}(n/6)^{1/3}\le \lam (5^{2h+1})^c<(n/6)^{1/3}.
\end{equation}
Thus from (\ref{3.14}) we have
$$n^{(\gam_0-2/3)/(2c)}(\log n)^{-10t^2K}<5^h<n^{(\gam_0-2/3)/(2c)}(\log n)^{10t^2K}.$$

\par Let $B$ be any fixed solution of (\ref{3.15}) with $1\le B\le 5^{2h}$. Then with $\nu$ either equal to $0$ or to $1$, one has
$$2\lam^2(B+5^{2h}\nu)^3-m\not\equiv 0\mmod{5^{2h+1}}.$$
The methods of Iwaniec \cite{Iwa1974}, Huxley \cite{Hux1975} and Gallagher \cite{Gal1972} may be employed to show that, whenever $Q>(5^{2h+1})^{12/5+\eps}$, then there is a prime number $p$ with $Q<p\le 2^{1/3}Q$, and satisfying
$$p\nmid m,\quad p\equiv B+5^{2h}\nu \mmod{5^{2h+1}}\quad \text{and}\quad p\equiv 1\mmod{3}.$$
We direct the reader to the discussion in the preamble to \cite[Theorem 2.1]{AGP1994}, as it relates to \cite[equation (0.3)]{AGP1994}, for an account of a suitable lower bound for this purpose. In addition, by virtue of the opening discussion of Heath-Brown \cite{HB1992}, on GRH such a prime exists whenever $Q>(5^{2h+1})^{2+\eps}$. We take $Q=\lam^{-1}(n/6)^{1/3}$, and observe from (\ref{3.16}) and the definition of $c$ that there is an $\eps$ to spare in the exponent of $5^{2h+1}$ to guarantee that $Q$ satisfies the respective lower bounds just cited. In this way we see that a suitable prime number $p$ exists with
\begin{equation}\label{3.17}
(n/6)^{1/3}<\lam p<(n/3)^{1/3},
\end{equation}
and with
$$p\equiv 1\mmod{3},\quad 2\lam^2p^3\equiv m\mmod{5^{2h}}\quad \text{and}\quad 2\lam^2p^3\not\equiv m\mmod{5^{2h+1}},$$
in which
\begin{equation}\label{3.18}
n^{(\gam_0-2/3)/c}(\log n)^{-20t^2K}<5^{2h}<n^{(\gam_0-2/3)/c}(\log n)^{20t^2K}.
\end{equation}

\par Write $N=TM^2$, where
$$T=5^{-2h}(m-2\lam^2p^3)\quad \text{and}\quad M=5^h.$$
Then we have $(T,M)=1$, and in view of (\ref{3.12}) we have yet only to solve the representation problem
$$N=x^2+y^2+6pz^2,$$
subject to the condition $\lam p>(n/6)^{1/3}$. Here, we may suppose that
$$N=m-2\lam^2p^3\equiv m\mmod{2},\quad N=m-2\lam^2p^3\equiv m-2\mmod{3}$$
and
$$(N,p)=(m-2\lam^2p^3,p)=(m,p)=1.$$
Since $m$ is odd, it follows that the congruence
$$N\equiv x^2+y^2+6pz^2\mmod{16}$$
is soluble. Moreover, we have $3\nmid (m-2)$ and $2\nmid m$, so that $(N,6p)=1$. Then it follows from Theorem \ref{theorem2.1} that this representation problem is soluble provided only that
$$NM^{12}>C(\eps)p^{21+\eps}\quad \text{and}\quad \lam p>(n/6)^{1/3}.$$
The second of these conditions is satisfied by virtue of (\ref{3.17}), whilst from (\ref{3.13}), (\ref{3.17}) and (\ref{3.18}) in combination with the latter constraint, we have
\begin{align*}
NM^{12}&>(m-2\lam^2p^3)(5^h)^{12}\\
&\ge \lam^{-1}n(1-2t(\log n)^{-1}-\tfrac{2}{3})(n^{(\gam_0-2/3)/c-\eps/8})^6.
\end{align*}
On recalling (\ref{3.14}), we discern that
\begin{equation}\label{3.19}
NM^{12}\ge n^{\gam_0+6(\gam_0-2/3)/c-\eps}.
\end{equation}
Meanwhile, from (\ref{3.17}) in combination with (\ref{3.14}) we find that
\begin{equation}\label{3.20}
p^{21+\eps}<(n^{1/3}\lam^{-1})^{21+\eps}<n^{21\gam_0-14+\eps}.
\end{equation}

Provided that
\begin{equation}\label{3.21}
\gam_0(6/c+1)-4/c-\eps>21\gam_0-14+\eps,
\end{equation}
then a comparison of (\ref{3.19}) and (\ref{3.20}) reveals that $NM^{12}>C(\eps)p^{21+\eps}$, and so it follows from the discussion of the last paragraph that we have a representation of $m$ in the shape (\ref{3.12}). This in turn implies that $m_t$ is represented in the form (\ref{3.1}), and hence that the Waring problem (\ref{1.4}) is indeed soluble. The inequality (\ref{3.21}) yields the condition
$$\gam_0<\frac{14-4/c-2\eps}{21-6/c-1}=\frac{7c-2-c\eps}{10c-3}.$$
Provided that $\gam_0<\tfrac{12}{17}$ and $\eps$ is chosen sufficiently small, therefore, this condition is satisfied with $c=2+2\eps$, which is the above choice of $c$ that follows from the assumption of GRH. Alternatively, as long as $\gam_0<\tfrac{74}{105}$ and $\eps$ is taken to be sufficiently small, this condition is satisfied with $c=\tfrac{12}{5}+2\eps$, which is our earlier choice for $c$ relevant when GRH is not assumed to hold. Thus, on recalling (\ref{1.3}) we find that when GRH holds and $\gam(\bfk)<\tfrac{12}{17}$, then we are able to solve the equation $N=x^2+y^2+6pz^2$, whence all sufficiently large natural numbers $n$ have a representation in the form (\ref{1.4}). The same conclusion holds independent of GRH when the $k_i$ are not all even and $\gam(\bfk)<\tfrac{74}{105}$. This completes our discussion of the central argument required to prove Theorem \ref{theorem1.2} when either the validity of GRH is assumed, or else when the exponents $k_1,\ldots ,k_t$ are not all even.\par

We now briefly sketch the modifications required to deliver the conclusion of Theorem \ref{theorem1.2} when $t\ge 2$ and $\gamtil(\bfk)<\tfrac{74}{105}$. In this instance, we modify the argument above following (\ref{3.3}) by applying Lemma \ref{lemma2.5} in place of Lemma \ref{lemma2.3}. We consider fixed distinct primes $\varpi_1,\ldots ,\varpi_{t-1}$ as before. It follows from the former lemma that there exist natural numbers $y_1=y_{t-1,t-1}$, $y_2=y_{t,t-1}$ and $h=h_{t-1}$ having the properties (a) to (d) of its conclusion. With the integers $y_{t-1,t-1}$, $y_{t,t-1}$ and $h_{t-1}$ fixed, we find that the representation problem (\ref{3.1}) may be solved whenever the representation problem
\begin{align*}
m_t-y_{t-1,t-1}^{k_{t-1}}-y_{t,t-1}^{k_t}=(\varpi_{t-1}^{3Kh_{t-1}}&v_{t-2})^2+(\varpi_{t-1}^{3Kh_{t-1}}w_{t-2})^2+
6(\varpi_{t-1}^{6Kh_{t-1}}A_{t-2})z^2\\
&+2(\varpi_{t-1}^{6Kh_{t-1}}A_{t-2})^3+\sum_{j=1}^{t-2}(\varpi_{t-1}^{6h_{t-1}K/k_j}y_{j,t-2})^{k_j}
\end{align*}
is soluble in natural numbers $v_{t-2},w_{t-2},A_{t-2},z$ and $y_{j,t-2}$ $(1\le j\le t-2)$, in which we impose the condition
$$\varpi_{t-1}^{6Kh_{t-1}}A_{t-2}>(n/6)^{1/3}.$$

\par Write
$$m_{t-2}=\varpi_{t-1}^{-6Kh_{t-1}}(m_t-y_{t-1,t-1}^{k_{t-1}}-y_{t,t-1}^{k_t}).$$
Then in view of property (a) of Lemma \ref{lemma2.5}, this representation problem is equivalent to (\ref{3.3}) subject to (\ref{3.4}) in the special case $r=t-2$, where we now write
$$\Ups_r=\prod_{r<l\le t-1}\varpi_l^{6Kh_l}.$$
Properties (a) to (d) of Lemma \ref{lemma2.5} deliver the relations (\ref{3.5}) to (\ref{3.9}), though one must now replace the exponents $\gam_r$ with the modified exponent
$$\gamtil_r=\left( 1-\frac{1}{k_t}\right) \prod_{r<l\le t-2}\left( 1-\frac{1}{k_l}\right).$$
The problem now assumes the form treated above, save that the exponents $\gam_r$ must be decorated by tildes throughout. The reader should experience no difficulty in completing the argument as before, concluding that Theorem \ref{theorem1.2} holds when $t\ge 2$ and $\gamtil(\bfk)<\frac{74}{105}$.\par

It now remains only to note that when $\gam(\bfk)<\tfrac{2}{3}$ (or indeed $\gamtil(\bfk)<\tfrac{2}{3}$), the above argument can of course be modified so that the exponents $h_i$ are made smaller. In this way, in the conclusion of Lemma \ref{lemma2.3}, one may ensure that property (b) is replaced by a condition of the type
$$n^{\ome/k_t}(\log n)^{-20tK}\le \varpi^{6Kh}\le y\le n^{\ome/k_t}(\log n)^{-t},$$
with similar modifications in Lemmata \ref{lemma2.4} and \ref{lemma2.5}. We choose $\ome=\ome(\nu)$ to be the positive number with $0<\ome<1$ for which the modified exponent
$$\gam(\bfk;\ome)=\prod_{j=1}^t\left( 1-\frac{\ome}{k_j}\right)$$
satisfies $\gam(\bfk;\ome)=\tfrac{2}{3}+\nu$, for a suitably small positive number $\nu$. By modifying in this way the central argument applied above, one derives the same conclusion save with $\gam(\bfk)$ replaced by $\gam(\bfk;\ome)$. Since $\tfrac{2}{3}<\gam(\bfk;\ome)=\tfrac{2}{3}+\nu$, the desired conclusion follows. Similar comments apply also in the situation discussed in the previous paragraph.\par

Theorem \ref{theorem1.1} follows immediately from Theorem \ref{theorem1.2} on noting that $(\tfrac{5}{6})^2<\tfrac{12}{17}$. Likewise, Corollaries \ref{corollary1.4} and \ref{corollary1.3} follow on noting that
$$\gam(5,8)=\left(\frac{4}{5}\right)\left(\frac{7}{8}\right)<\frac{74}{105},\quad \gam(9,9,9)=\left(\frac{8}{9}\right)^3<\frac{74}{105},$$
$$\gam(6,12,12)=\left(\frac{5}{6}\right)\left(\frac{11}{12}\right)^2<\frac{12}{17}.$$

\section{Consequences of the Ramanujan Conjecture}
In this section we refine the conclusion of Theorem \ref{theorem2.1}, assuming the truth of the Ramanujan Conjecture concerning Fourier coefficients of cusp forms of weight $\tfrac{3}{2}$. 

\begin{theorem}\label{theorem4.1}
Assume the truth of the Ramanujan Conjecture. For each $\del>0$, there exists a positive number $C(\del)$ with the property that whenever $n\in \dbN$ and $p$ is a prime number satisfying the following conditions:
\item{(i)} one has $(n,6p)=1$;
\item{(ii)} the congruence $n\equiv x^2+y^2+6pz^2\mmod{16}$ is soluble;
\item{(iii)} one has $n>C(\del)p^{5+\del}$;
\vskip.1cm \noindent then the equation (\ref{2.2}) has a solution in natural numbers $x$, $y$ and $z$.
\end{theorem}

\begin{proof} We follow the argument of \cite[\S2]{Gol2008}, mutatis mutandis. Write $r(n)$ for the number of representations of the integer $n$ in the shape (\ref{2.2}) with $x,y,z\in \dbN$. Denote by $S_{3/2}^{(1)}(24p,\chi)$ the space of cusp forms of weight $\tfrac{3}{2}$ and level $24p$ which, under the Shimura lift, are taken into cusp forms of weight $2$. Let $K$ denote the dimension of $S_{3/2}^{(1)}(24p,\chi)$, so that $K\ll p^{1+\eps}$, and let $(\varphi_k(z))_{k=1}^K$ be an orthonormal basis for this space.\par

Let $\fhat(n)$ be the $n$-th Fourier coefficient of a cusp form $f(z)$, so that
$$f(z)=\sum_{n=0}^\infty \fhat(n)e(nz).$$
We write $n$ in the shape $n=tm^2$ with $t$ squarefree. Note that the condition $(n,6p)=1$ ensures that both $t$ and $m$ are coprime to the level. Then by combining Deligne's estimates for the eigenvalues of Hecke operators of weight $2$, together with Shimura's lift, one finds that
$$\phihat_k(n)\ll |\phihat_k(t)|m^{1/2+\eps}.$$
This estimate is employed in the argument of the proof of \cite[Theorem 2]{Gol2008}. We direct the reader to \cite[Lemma 1]{Blo2008} for a discussion of this conclusion. The Ramanujan Conjecture asserts that when $f\in S_{3/2}^{(1)}(24p,\chi)$, then for squarefree integers $r$ one should have
$$|\fhat(r)|\ll r^{1/4+\eps}.$$
We apply this estimate trivially to obtain
$$\sum_{k=1}^K|\phihat_k(t)|^2\ll Kt^{1/2+2\eps}\ll p^{1+\eps}t^{1/2+2\eps}.$$
Recall that $n=tm^2$. Then on substituting this bound for Golubeva's use of the Duke-Iwaniec estimate
$$\sum_{k=1}^K|\phihat_k(t)|^2\ll t^{13/14+\eps},$$
the reader will have no difficulty in following the argument of \cite[\S2]{Gol2008} so as to establish that
$$r(n)\gg_\eps n^{1/2-\eps}p^{-1/2}+O\left(p^{1/4+\eps}m^{1/2+\eps}(p^{1+\eps}t^{1/2+2\eps})^{1/2}\right).$$
Consequently, one has $r(n)\gg_\eps n^{1/2-2\eps}p^{-1/2}$ provided only that
$$n^{1/2-2\eps}p^{-1/2}\gg p^{3/4+2\eps}(tm^2)^{1/4+\eps},$$
or equivalently,
$$n^{1/4-3\eps}\gg p^{5/4+2\eps}.$$
Then, whenever $\del>0$, and $\eps>0$ is taken sufficiently small in terms of $\del$, one finds that 
$r(n)\gg_\eps n^{1/2-2\eps}p^{-1/2}$ provided only that $n\gg p^{5+\del}$. This completes the 
proof of the theorem.
\end{proof}

The conclusion of Theorem \ref{theorem4.1} may be substituted into the argument of \S3 without any great difficulty. The lower bound (\ref{3.19}) is replaced by $N\ge n^{\gam_0-\eps}$, whilst (\ref{3.20}) becomes
$$p^{5+\eps}<(n^{1/3}\lam^{-1})^{5+\eps}<n^{5\gam_0-10/3+\eps}.$$
Hence, provided that
\begin{equation}\label{4.1}
\gam_0-\eps>5\gam_0-\tfrac{10}{3}+\eps,
\end{equation}
a comparison reveals that $N>p^{5+\eps}$, and so it follows as before that the Waring problem (\ref{1.4}) 
is soluble, but now subject to the truth of the Ramanujan Conjecture. The inequality (\ref{4.1}) yields the condition
$$\gam_0<\tfrac{1}{4}(\tfrac{10}{3}-2\eps)=\tfrac{5}{6}-\tfrac{1}{2}\eps.$$
Provided that $\gam_0<\tfrac{5}{6}$ and $\eps$ is chosen sufficiently small, therefore, this condition is satisfied. All that remains is to note that in the suppressed argument leading to this point in our discussion 
we find it necessary to assume GRH only when all the exponents $k_1,\ldots ,k_t$ are even. The variants sketched at the end of \S3 follow with obvious modifications that need not detain us here. This completes the proof of Theorem \ref{theorem1.5}.\par

We finish by noting that a more optimistic enhancement of Theorem \ref{theorem4.1} might be expected to hold with the condition (iii) replaced by the constraint that $n>C(\del)p^{3+\del}$. If true, such would replace the condition (\ref{4.1}) by
$$\gam_0-\eps>3(\gam_0-\tfrac{2}{3})+\eps.$$
This condition is satisfied provided that $\gam_0<1$ and $\eps$ is chosen sufficiently small. Such a conclusion would imply that when $n$ is sufficiently large, the Waring problem (\ref{1.4}) is soluble whenever either $t\ge 2$, or the exponents $k_1,\ldots ,k_t$ are not all even. When $t=1$ and $k_1$ is even, the Waring problem (\ref{1.4}) would remain soluble provided that GRH holds.

\bibliographystyle{amsbracket}
\providecommand{\bysame}{\leavevmode\hbox to3em{\hrulefill}\thinspace}

\end{document}